\documentclass{amsart}

\makeatletter
 \def\LaTeX{\leavevmode L\raise.42ex
   \hbox{\kern-.3em\size{\sf@size}{0pt}\selectfont A}\kern-.15em\TeX}
\makeatother

\newcommand{\BibTeX}{{\rm B\kern-.05em{\sc
i\kern-.025emb}\kern-.08em\TeX}}
\newtheorem{col}{Corollary}[section]
\newtheorem{thm}{Theorem}[section]
\newtheorem{lem}[thm]{Lemma}

\theoremstyle{definition}
\newtheorem{defn}{Definition}

\numberwithin{equation}{section}

\begin{document}

\title{Cubature formulas on combinatorial graphs}
\maketitle

\begin{center}

\author{Isaac Z. Pesenson }\footnote{ Department of Mathematics, Temple University,
 Philadelphia,
PA 19122; pesenson@math.temple.edu. The author was supported in
part by the National Geospatial-Intelligence Agency University
Research Initiative (NURI), grant HM1582-08-1-0019. }

\author{Meyer Z. Pesenson }\footnote{ 
CMS Department, California Institute of Technology, MC 305-16, Pasadena, CA 91125; mzp@cms.caltech.edu.
The author was supported in part by the National Geospatial-Intelligence Agency University Research Initiative (NURI), grant HM1582-08-1-0019 and by AFOSR, MURI, Award FA9550-09-1-0643  }

\author{Hartmut F\"{u}hr}
\footnote{Lehrstuhl A f\"ur Mathematik, RWTH Aachen, D-52056 Aachen, Germany; fuehr@matha.rwth-aachen.de}

\end{center}

\begin{abstract}
 The goal of the paper is to establish cubature formulas on combinatorial graphs.
Two types of cubature formulas are developed.  Cubature formulas of the first type are exact on spaces of variational splines on graphs. Since badlimited functions can be obtained as limits of variational splines we obtain cubature formulas which are "essentially"   exact on spaces of bandlimited functions.  Cubature formulas of the second type are exact on spaces of bandlimited functions.  Accuracy  of   cubature formulas is given in terms of smoothness which is measured by means of combinatorial Laplace operator. The results have potential  applications to problems that arise in data  mining.

\end{abstract}



 \section{Introduction}

Cubature formulas for approximate and exact evaluation of integrals of functions defined on Euclidean spaces or on smooth manifolds is an important and continuously developing subject. During last years in connection with applications to information theory analysis of functions defined on combinatorial graphs attracted attention of many mathematicians.  The following list of a few classical and recent papers is very far from being complete: \cite{CM}, \cite{GNC}, \cite{HVG}, \cite{SNFOV}, \cite{NSZ}, \cite{NO}.

In particular certain cubature formulas for functions defined on combinatorial graphs was recently considered in the paper \cite{GNC}. There, given values of a function f on a small subset $U$ of the set of all vertices $V$ of a graph, the authors estimate wavelet coefficients via specific  cubature formulas.

In the present paper we develop a set of rules (cubature formulas)  which allow for approximate or exact evaluation of "integrals" $\sum_{v\in V}f(v)$ of functions  by using their values  on  subsets  $U\subset V$ of vertices. 
We make extensive use  of our previous work on Shannon sampling of bandlimited functions  and variational splines on combinatorial graphs \cite{FP}-\cite{PesPes2}.
Our results can find applications to problems that arise in connection with data
  filtering, data denoising and  data dimension reduction.

In section 2 we review our results \cite{Pes09} about variational interpolating spines on graphs and describe an algorithm which allows an effective computation of variational splines. In section 3 by using interpolating splines  we develop a set of cubature formulas. Theorem 3.1 shows that these formulas are \textit{exact} on the set of variational splines. Theorem \ref{opt} explains that our cubature formulas are \textit{optimal}. 

In section 4, using our result that  bandlimited functions are limits of variational splines (see  \cite{Pes08}, \cite{Pes09}) we show, that cubature formulas developed in section 3 are "essentially" exact on bandlimited functions.

It can be verified for example,   that for a cycle graph of 1000 vertices a set  of  about $670$ "uniformly" distributed vertices is sufficient to have asymptotically exact cubature formulas for linear combinations of the first $290$ eigenfunctions (out of $1000$) of the corresponding combinatorial Laplace operator.

It is worth to note that all results of section 3 which provide errors of approximation of integrals of functions on $V$ through their values on a $U\subset V$  reflect 

1) geometry of $U$ which is  inherited  into the quantity $\sqrt{|V|-|U|}=\sqrt{|S|}$ and into the Poincare constant $\Lambda$ (see section 3 for definitions),

2) smoothness of functions which is measured in terms of combinatorial Laplace operator.

In section 4 we develop a different set of cubature formulas which are \textit{exact} on appropriate sets of bandlimited functions. The results in this section are formulated in the language of frames and only useful if it is possible to calculate dual 
frames explicitly. Since in general it is not easy to compute a dual frame we finish this section by explaining another approximate cubature formula which is based on the so-called frame algorithm.

This paper is a discrete counterpart of the paper \cite{PG}.
In a  forthcoming paper we are going to extend our results to weighted and infinite graphs.

\section{Variational (polyharmonic) splines on graphs}

Let $G$ denote an undirected weighted graph, with  a finite or countable number of vertices $V(G)$  and weight function $w: V(G) \times V(G) \to \mathbb{R}_0^+$. $w$ is symmetric, i.e., $w(u,v) = w(v,u)$, and $w(u,u)=0$ for all $u,v \in V(G)$. The edges of the graph are the pairs $(u,v)$ with $w(u,v) \not= 0$.

Let $L_{2}(G)\>\>$ denote the space of  all real-valued functions with the inner product
$$
\left<f,g\right>=\sum_{v\in V(G)}f(v)g(v)
$$
and  the norm
\[
 \| f \| = \left( \sum_{v \in V(G)} |f(v)|^2 \right)^{1/2}.
\]

In  the case of a finite graph and $L_{2}(G)$-space  the weighted Laplace operator $\mathcal{L}: L_{2}(G) \to L_{2}(G)$  is introduced  via
\begin{equation}\label{L}
 (\mathcal{L} f)(v) = \sum_{u \in V(G)} (f(v)-f(u)) w(v,u)~.
\end{equation}
This  graph Laplacian is a well-studied object; it is known to be a positive-semidefinite self-adjoint \textit{bounded} operator. 
The notation
$l_{2}$ will be used for the Hilbert space of all sequences of
real numbers $\overline{y}=\{y_{\nu}\}$, for which $ \sum_{\nu}|y_{\nu}|^{2}< \infty. $
\bigskip

Variational splines on combinatorial graphs were developed in \cite{Pes09}. 

 \textbf{Variational Problem}

 Given a  subset of vertices $U=\{u\}\subset V,  $
 a sequence of real numbers
$\overline{y}=\{y_{u}\}\in l_{2}, u\in U$, a natural $k$, and a positive $\varepsilon>0$ we consider the following
variational problem:

\bigskip

 \textsl{Find a function  $Y$ from the
space $L_{2}(G)$ which has the following properties:}

1) $Y(u)=y_{u}, u\in U,$

2) $Y$ \textsl{minimizes functional $Y\rightarrow
\left\|(\varepsilon I+\mathcal{L}) ^{k}Y\right\|$.}

We show that the above variational problem has  a unique solution
$Y_{k,\varepsilon}^{U,\overline{y}}$.  

For the sake of simplicity we will also use notation
$Y_{k}^{\overline{y}}$ assuming that $U$ and $\varepsilon$ are fixed.

We say that $Y_{k}^{\overline{y}}$
is a\textit{ variational  spline }of order $k$.  It is also shown
that every spline is a linear combination of fundamental solutions
of the operator $(\varepsilon I+\mathcal{L})^{k}$ and in this
sense it is a \textit{polyharmonic } function with singularities. Namely it
is shown that every spline satisfies the following equation
\begin{equation}
\label{rep1}
(\varepsilon I+\mathcal{L})^{2k}Y_{k}^{\overline{y}}=\sum_{u\in
U}\alpha_{u}\delta_{u},
\end{equation}
where $\{\alpha_{u}\}_{u\in
U}=\left\{\alpha_{u}(Y_{k}^{\overline{y}})\right\}_{u\in U}$ is a sequence
from $l_{2}$ and $\delta_{u}$ is the Dirac measure at a vertex
$u\in U$. The set of all such splines for a fixed $U\subset V$
and fixed $k>0, \varepsilon\geq 0,$ will be denoted as
$\mathcal{Y}(U, k, \varepsilon).$

A fundamental solution  $F_{k}^{u} \left(=F_{k,\varepsilon}^{u}\right), u\in V,$ of
the operator $(\varepsilon I+\mathcal{L})^{k} $ is the solution of
the equation

\begin{equation}
(\varepsilon I+\mathcal{L})^{k}F_{k}^{u}=\delta_{u},\>k\in \mathbb{N},
\end{equation}
where $\delta_{u}$ is the Dirac measure at $u\in V(G)$.
It follows from (\ref{rep1}) that the following representation holds
$$
Y_{k}^{\overline{y}}=\sum_{u\in
U}\alpha_{u}F^{u}_{2k}.
\label{fund.sol.representation}
$$

It is shown in \cite{Pes09} that for every  set of vertices
$U=\{u\}, $ every natural $k$,  every $\varepsilon \geq 0,$ and for any given
sequence $\overline{y}=\{y_{u} \}\in l_{2},$ the solution
$Y_{k}^{\overline{y}}$ of the Variational Problem has a
representation
\begin{equation}
\label{repren1}
Y_{k}^{\overline{y}}=\sum_{u\in U} y_{u}L^{u}_{k},
\end{equation}
where $L^{u}_{k}$ is the so called Lagrangian
spline, i.e. it is a solution of the same Variational Problem with
constraints $ L^{u}_{k}(v)=\delta_{u,v},\> u\in U,$
where $\delta_{u,v}$ is the Kronecker delta. It implies in particular, that
$\mathcal{Y}(U, k, \varepsilon)$ is a linear set.

Given a function $f\in L_{2}(G)$ we will say that the spline
$Y_{k}^{f}$  interpolates $f$ on $U$ if
$Y_{k}^{f}(u)=f(u)$ for all $u\in U$.

\bigskip

\textbf{Algorithm for computing variational splines.}

The above results give a constructive way for computing
variational splines.  Suppose we are going to
construct splines which have prescribed values on a subset of vertices $U\subset V$.

1.  One has to solve the following $|U|$ systems of linear equations  of the size $|V|\times|V|$
\begin{equation}
(\varepsilon I+\mathcal{L})^{k}F_{k,\varepsilon}^{u}=\delta_{u},
u\in U, \>\>\>k\in \mathbb{N},
\end{equation}
in order to determine functions $F_{k,\varepsilon}^{u}$.

2. Let $ \delta_{w,v}$ be
the Kronecker delta.  One has to solve  $|U| $ linear
system of the size  $|U|\times |U| $ to determine coefficients $\alpha_{u}^{w}$

\begin{equation}
\delta_{w, \gamma}=\sum_{u\in U}\alpha_{u}^{w} F_{k,\varepsilon}^{u}(\gamma)              ,  \>\>\>\> w, \gamma\in U.
\end{equation}

3. It gives
the following representation of the corresponding Lagrangian spline
\begin{equation}
L^{w}_{k, \varepsilon}=\sum_{u\in U}\alpha_{u}^{w}F_{k,\varepsilon}^{u},\>\>\>\> w\in U.
\end{equation}

4. Every spline $Y_{s,\varepsilon}^{y}\in
\mathcal{Y}(U,s,\varepsilon)$ which takes prescribed values
$\overline{y}=\{y_{w} \}, w\in U,$ can be written
explicitly as
$$ Y_{s,\varepsilon}^{y}=\sum_{w\in W} y_{w}L^{w}_{s,
\varepsilon}. 
$$

\section{Cubature formulas  which are exact  on variational splines  }

 We introduce the following  scalars
$$
\theta_{u}=\theta_{u}(U, \>k,\>\varepsilon)=\sum_{v\in V}L_{k,\varepsilon}^{U, u}(v)
$$
and by applying the formula (\ref{repren1}) we obtain the following fact.

\begin{thm}
\label{ExactS}
In the same notations as above for every subset of vertices $U=\{u\}$  and every $k\in \mathbb{N}, \varepsilon > 0,$ there exists a set of weights $\theta_{u}=\theta_{u}(U, \>k,\>\varepsilon),\>u\in U,$ such that for every spline $Y_{k,\varepsilon}^{U,y}$ that takes values $Y_{k,\varepsilon}^{U,y}(u)=y_{u},\>u\in U,$ the following exact formula holds
\begin{equation}
\label{exact}
\sum_{v\in V}Y_{k,\varepsilon}^{U,y}(v)=\sum_{u\in U}y_{u}\theta_{u}
\end{equation}

\end{thm}

 For a subset $S\subset V$ (finite or infinite) the
notation $L_{2}(S)$ will denote the space of all functions from
$L_{2}(G)$ with support in $S$:
$$
L_{2}(S)=\{\varphi\in L_{2}(G), \>\varphi(v)=0, \>v\in V\backslash
S\}.
$$
\begin{defn}
We say that a set of vertices $S\subset V$ is a $\Lambda$-set
if for any $\varphi\in L_{2}(S)$ it admits a Poincare  inequality
with a constant $\Lambda=\Lambda(S)>0$
\begin{equation}
\|\varphi\|\leq \Lambda\|\mathcal{L}\varphi\|,\> \varphi\in
L_{2}(S).
\end{equation}
The infimum of all $\Lambda>0$ for which $S$ is a $\Lambda$-set
will be called the Poincare constant of the set $S$ and denoted by
$\Lambda(S)$.
\end{defn}

The following lemma holds true \cite{Pes09}.

\begin{lem}
\label{lemma}
If $A$ is a  self-adjoint positive definite operator in a
Hilbert space $H$ and for an  $\varphi\in H$ and a positive $a>$
the following inequality holds true
$$
\|\varphi\|\leq a\|A\varphi\|,
$$ then for the same $\varphi \in H$, and all $ k=2^{l}, l=0,1,2,...$ the following
inequality holds
$$
\|\varphi\|\leq a^{k}\|A^{k}\varphi\|.
$$
\end{lem}

\begin{proof}
By the spectral theory there exist a direct integral of
Hilbert spaces
$$
X=\int_{0}^{\|A\|} X(\tau)dm (\tau )
$$
 and a unitary operator
$F$ from $H$ onto $X$, which transforms domain of $A^{t}, t\geq
0,$ onto $X_{t}=\{x\in X|\tau^{t}x\in X \}$ with norm
$$
\|A^{t}f\|_{H}=\left (\int_{0}^{\|A\|} \tau^{2t} \|Ff(\tau
)\|^{2}_{X(\tau)} d m(\tau) \right )^{1/2}
$$
and $F(A^{t} f)=\tau ^{t} (Ff)$. According to our assumption we
have for a particular  $\varphi\in H$
$$
\int _{0}^{\|A\|}| F\varphi(\tau)|^{2}d m(\tau)\leq a^{2} \int
_{0}^{\|A\|}\tau ^{2}| F\varphi(\tau)|^{2}d m(\tau)
$$
and then for the interval $B=B(0, a^{-1})$ we have
$$\int _{B}| F\varphi(\tau)|^{2}d m(\tau)+
\int_{[0, \|A\|]\setminus B}|F\varphi|^{2}d m(\tau)\leq
$$
$$
a^{2}\left( \int_{B}\tau^{2}|F\varphi|^{2}d m(\tau) +\int_{[0,
\|A\|]\setminus B}\tau^{2}|F\varphi| ^{2}d m(\tau)\right ) .
$$
 Since $a^{2}\tau^{2}<1$ on $B(0, a^{-1})$
$$
0\leq \int_{B}\left
(|F\varphi|^{2}-a^{2}\tau^{2}|F\varphi|^{2}\right)d m(\tau) \leq
\int _{[0, \|A\|]\setminus B}\left ( a^{2}
\tau^{2}|F\varphi|^{2}-|F\varphi|^{2}\right)d m(\tau).
$$
This inequality
 implies the inequality
 $$
 0\leq \int_{B}\left (a^{2}\tau^{2}|F\varphi|^{2}-
a^{4}\tau^{4}|F\varphi|^{2}\right)d m(\tau) \leq \int_{[0,
\|A\|]\setminus B}\left ( a^{4}
\tau^{4}|F\varphi|^{2}-a^{2}\tau^{2}|F\varphi| ^{2}\right)d
m(\tau)$$
 or
  $$a^{2}\int_{[0, \|A\|]}\tau^{2}|F\varphi|^{2}d m(\tau) \leq
 a^{4}\int_{\mathbb{R}_{+}}\tau^{4}|F\varphi|
 ^{2}d m(\tau),
 $$
 which means
 $$
 \|\varphi\|\leq a\|A\varphi\|\leq a^{2}\|A^{2}\varphi\|.
 $$
 Now, by using induction one can finish the proof of the
Lemma. The Lemma is proved.
\end{proof}

The following Theorem gives a cubature rule that allows to compute the integral $\sum_{v\in V}f(v)$ by using only values of $f$ on a smaller set $U$.

\begin{thm}
\label{fund}
For every set of vertices   $U\subset V$  for which $S=V\setminus U$ is a $\Lambda$-set  and for any $\varepsilon>0, \> k=2^{l}, \>l\in \mathbb{N},$ there exist weights $\theta_{u}=\theta_{u}(U, \>k,\>\varepsilon)$ such that for
every function $f\in L_{2}(G),$
\begin{equation}
\label{cub}
\left |\sum_{v\in V}f(v)-\sum_{u\in U}f(u)\theta_{u}\right |\leq 2\sqrt{|S|}\Lambda^{k}
\left\|\left(\varepsilon I+\mathcal{L}\right)^{k}f\right\|.
\end{equation}
\end{thm}

\begin{proof}

 If $f\in L_{2}(G)$ and
$Y_{k,\varepsilon}^{U, f}$ is a variational spline which
interpolates $f$ on a set $U=V\setminus S$ then
\begin{equation}
\label{cub1}
\left |\sum_{v\in V}f(v)-\sum_{v\in V}Y_{k,\varepsilon}^{U, f}(v)\right |\leq\sum_{v\in S}\left|f(v)-Y_{k,\varepsilon}^{U, f}(v)\right |\leq \sqrt{|S|}\left\|f-Y_{k,\varepsilon}^{U, f}\right\|
\end{equation}
Since $S$  is a
$\Lambda$- set we have
\begin{equation}
\|f-Y_{k,\varepsilon}^{U, f}\left\|\leq
\Lambda \| \mathcal{L}\left(f-Y_{k,\varepsilon}^{U, f}\right)\right\|.
\end{equation}
For any $g\in L_{2}(G)$ the following inequality holds true
\begin{equation}
\label{exten}
\|\mathcal{L}g\|\leq \|(\varepsilon I+\mathcal{L})g\|.
\end{equation}
Thus one obtains the inequality
\begin{equation}
\label{ineq1}
\left\|f-Y_{k,\varepsilon}^{U, f}\right\|\leq \Lambda\left\|(\varepsilon
I+\mathcal{L})\left(f-Y_{k,\varepsilon}^{U, f}\right)\right\|.
\end{equation}

We apply  Lemma  \ref{lemma}  with $A=\varepsilon I+\mathcal{L}$,
$a=\Lambda$ and $\varphi=f-Y_{k,\varepsilon}^{U, f}$. It gives the
inequality
\begin{equation}
\label{LA}
\left\|f-Y_{k,\varepsilon}^{U, f}\right\|\leq \Lambda^{k}\left\|(\varepsilon
I+\mathcal{L})^{k}\left(f-Y_{k,\varepsilon}^{U, f}\right)\right\|
\end{equation}
for all $ k=2^{l}, l=0,1,2,...$
 Using the minimization property of $Y_{k,\varepsilon}^{U, f}$ we
obtain
$$
\left\|f-Y_{k,\varepsilon}^{U, f}\right\|\leq 2\Lambda^{k}\left\|\left(\varepsilon
I+\mathcal{L}\right)^{k}f\right\|,k=2^{l}, l\in \mathbb{N}.
$$
Together with (\ref{cub1}) it gives
\begin{equation}
\left |\sum_{v\in V}f(v)-\sum_{v\in V}Y_{k,\varepsilon}^{U, f}(v)\right |\leq 2\sqrt{|S|}\Lambda^{k}\|(\varepsilon
I+\mathcal{L})^{k}f\|,\>k=2^{l}, l\in \mathbb{N}.
\end{equation}
By applying the Theorem \ref{ExactS} we finish the proof.
\end{proof}

It is worth to note that the above formulas are optimal in the sense it is described below.

\begin{defn}
For the given  $U\subset V, f\in L_{2}(G), k\in \mathbb{N}, \varepsilon\geq
0, R>0,$ the notation $Q(U,f,k, \varepsilon, R)$ will be used for
a set of all functions $h$ in $L_{2}(G)$ such that

\bigskip

1) $h(u)=f(u), u\in U,$

\bigskip
and

\bigskip

2) $\left\|(\varepsilon I+\mathcal{L})^{k}h\right\|\leq R.$

\end{defn}

 It is easy to verify that every set $Q(U,f,k,
\varepsilon, R)$ is convex, bounded, and closed. It implies that the set of all integrals of functions in $Q(U,f,k,\varepsilon, R)$ is an interval i. e.
\begin{equation}
\label{int}
[a,\>b]=\left\{ \sum_{v\in V}h(v):\>\> h\in       Q(U,f,k,\varepsilon, R)     \right\}
\end{equation}

The optimality result is the following.

\begin{thm}
\label{opt}
For every set of vertices   $U\subset V$  and for any $\varepsilon>0, \> k=2^{l}, \>l\in \mathbb{N},$ if  $\theta_{u}=\theta_{u}(U, \>k,\>\varepsilon)$ are the same weights that appeared  in the previous statements, then for any $ g\in       Q(U,f,k,\varepsilon, R)  $ 
\begin{equation}
\sum_{u\in U}g(u)\theta_{u}=\frac{a+b}{2},
\end{equation}
where $[a,\>b]$ is defined in  (\ref{int}).
\end{thm}
\begin{proof}
We are going to show that for a given function $f$ the
interpolating spline $Y_{k, \varepsilon}^{U,f}$ is the center of
the convex, closed and bounded set $Q(U,f,k, \varepsilon, R)$ for
any $R\geq \left\|(\varepsilon I +\mathcal{L})^{k}Y_{k,\varepsilon}^{U,f}\right\|$ .
In other words it is sufficient to show that if
$$
Y_{k, \varepsilon}^{U,f}+h\in Q(U,f,k, \varepsilon, R)
$$
for some function $h$ then the function $Y_{k, \varepsilon}^{U,f}-h$
also belongs to the same intersection.
 Indeed, since $h$ is zero on the set $U$ then according to (\ref{rep1}) one has
$$
\left <(\varepsilon I+ \mathcal{L})^{k}Y_{k, \varepsilon}^{U,f},
(\varepsilon I+\mathcal{L})^{k}h \right>=\left <(\varepsilon I+
\mathcal{L})^{2k}Y_{k, \varepsilon}^{U,f}, h \right>=0.
$$
 But then

$$\left\|(\varepsilon I+\mathcal{L})^{k}(Y_{k, \varepsilon}^{U,f}+h)\right\|=
\left\|(\varepsilon I+\mathcal{L})^{k}\left(Y_{k,
\varepsilon}^{U,f}-h\right)\right\|.
$$
In other words,
 $$\left\|(\varepsilon I+\mathcal{L})^{k}\left(Y_{k, \varepsilon}^{U,f}-h\right)\right\|\leq R $$
 and because $Y_{k, \varepsilon}^{U,f}+h$ and $Y_{k, \varepsilon}^{U,f}-h$ take the same
values on $U$ the function $Y_{k, \varepsilon}^{U,f}-h$ belongs to
 $Q(U,f,k, \varepsilon, R).$ From here the Theorem follows. 
 \end{proof}

\begin{col}
Fix a function $f\in L_{2}(G)$ and a 
 set of vertices   $U\subset V$  for which $S=V\setminus U$ is a $\Lambda$-set. Then for any $\varepsilon>0, \> k=2^{l}, \>l\in \mathbb{N},$ for the same set of weights $\theta_{u}=\theta_{u}(U, \>k,\>\varepsilon)\in \mathbb{R}$ that appeared in the previous statements the following inequalities hold  for
every function $g\in Q(U,f,k, \varepsilon, R),$
\begin{equation}
\label{last}
\left |\sum_{v\in V}g(v)-\sum_{u\in U}f(u)\theta_{u}\right |\leq \sqrt{|S|}\Lambda^{k}diam  \>Q(U,f,k,\varepsilon, R).
\end{equation}
\end{col}
\begin{proof}
 Since $f$ and $g$ coincide on $U$ from (\ref{cub1}) and  (\ref{LA}) we obtain the inequality

\begin{equation}
\left |\sum_{v\in V}g(v)-\sum_{v\in V}Y_{k,\varepsilon}^{U, f}(v)\right |\leq   \sqrt{|S|}\Lambda^{k}\left\|(\varepsilon
I+\mathcal{L})^{k}\left(f-Y_{k,\varepsilon}^{U, f}\right)\right\|
\end{equation}
By the Theorem \ref{opt} the following inequality holds
$$
  \left\|(\varepsilon I+\mathcal{L})^{k}\left(Y_{k, \varepsilon}^{U,f}-g\right)\right\|\leq \frac{1}{2} diam \> Q(U,f,k,\varepsilon, R)
$$
for any $g\in Q(U,f,k,\varepsilon, R)$. The last two inequalities imply the Corollary.  
\end{proof}
\bigskip

\section{Approximate cubature formulas for bandlimited functions}

Operator(matrix) $\mathcal{L}$ is symmetric and positive definite. 
Let  $\textbf{E}_{\omega}(\mathcal{L})$ be the span of eigenvectors of $\mathcal{L}$
whose corresponding eigenvalues are $\leq \omega$. The invariant subspace $\textbf{E}_{\omega}(\mathcal{L})$
is the space of all vectors in $L_{2}(G)$ on which  $\mathcal{L}$ has norm $\omega$.
In other words $f$ belongs to $\textbf{E}_{\omega}(\mathcal{L})$ if and only if the following Bernstein-type inequality holds
\begin{equation}
\label{Bern}
\|\mathcal{L}^{s}f\|\leq \omega^{s}\|f\|, \> s\geq0.
\end{equation}
The Bernstein inequality (\ref{Bern}), the Lemma \ref{lemma}, and the Theorem \ref{fund} imply the following result.

\begin{col}  
For every set of vertices   $U\subset V$  for which $S=V\setminus U$ is a $\Lambda$-set  and for any $\varepsilon>0, \> k=2^{l}, \>l\in \mathbb{N},$ there exist weights $\theta_{u}=\theta_{u}(U, \>k,\>\varepsilon)\in \mathbb{R}$ such that for
every  function $f\in \textbf{E}_{\omega}(\mathcal{L}),$ the following inequality holds

\begin{equation}
\label{asymtexact}
\left |\sum_{v\in V}f(v)-\sum_{u\in U}f(u)\theta_{u}\right |\leq 2\gamma^{k}\sqrt{|S|}\left\|f\right\|,
\end{equation}
where $\gamma=\Lambda( \omega+\varepsilon),\> k=2^{l},\> l\in \mathbb{N}.$
\end{col}

If in addition the following condition holds 
$$
0<\omega<\frac{1}{\Lambda}-\varepsilon
$$
 and $f\in \textbf{E}_{\omega}(\mathcal{L})$ then this Corollary imply the following Theorem.

\begin{thm}  
\label{bd}
If $U$ is a subset of vertices  for which $S=V\setminus U$ is a $\Lambda$-set then for any $0<\varepsilon <1/\Lambda,\> k=2^{l}, \>l\in \mathbb{N},$ there exist weights $\theta_{u}=\theta_{u}(U, \>k,\>\varepsilon)\in \mathbb{R}$ such that for
every  function $f\in \textbf{E}_{\omega}(\mathcal{L}),$ where
$$
0<\omega<\frac{1}{\Lambda}-\varepsilon,
$$
the following relation holds

\begin{equation}
\label{asymtexact}
\left |\sum_{v\in V}f(v)-\sum_{u\in U}f(u)\theta_{u}\right |\rightarrow 0,
\end{equation}
when  $ k=2^{l}\rightarrow \infty.$
\end{thm}

\textbf{Example 1.}

Consider the unweighted cycle graph  $C_{1000}$ of $1000$ vertices. The Laplace operator $\mathcal{L}$  has one thousand eigenvalues which are given by the formula $\lambda_{k}=2-2\cos \frac{2\pi k}{1000},\> k=0,1,...,999$ (see \cite{Ch}).

It is easy to verify that every single vertex in $C_{1000}$ is a $\Lambda=\frac{1}{\sqrt{6}}$-set. It is also easy to understand that if closures of two vertices do not intersect i. e. 

$$\left(v_{j}\cup \partial v_{j}\right)\cap \left(v_{i}\cup \partial v_{i}\right)=\emptyset,\>\>v_{j}, v_{i}\in C_{1000},
$$ 
(here $\partial$ is the vertex boundary operator) then their union $v_{j}\cup v_{i}$ is also a $\Lambda=\frac{1}{\sqrt{6}}$-set. It implies, that one can remove from $C_{1000}$ every third  vertex and on the remaining set of $670$ the  formula (\ref{asymtexact}) will be true for the span of about $290$ first eigenfunctions of $\mathcal{L}$.

\bigskip
\textbf{Example 2.}

One can  show \cite{Pes08} that if $S=\{v_{1},v_{2},...,v_{N}\}$ consists of $|S|$ successive
vertices of the graph $C_{1000}$ then it is a $\Lambda$-set
with
$$
\Lambda=\frac{1}{2}\left(\sin\frac{\pi}{2|S|+2}\right)^{-2}.
$$
It implies for example that on a set of $100$ uniformly distributed vertices of $C_{100}$ the formula (\ref{asymtexact}) will be true for every function in the span of about $40$ first eigenfunctions of $\mathcal{L}$.

\bigskip

\section{Another set of exact and approximate cubature formulas for bandlimited functions}

We introduce another set of cubature formulas which are exact on some sets of bandlimited functions.

\begin{thm}
\label{exact1}
If  $U$ is a subset of vertices  for which $S=V\setminus U$ is a $\Lambda$-set then there exist weights $\sigma_{u}=\sigma_{u}(U)\in \mathbb{R},\>u\in U,$ such that for
every function $f\in \textbf{E}_{\omega}(\mathcal{L}),$ where
$$
0<\omega<\frac{1}{\Lambda},
$$
the following exact formula holds
\begin{equation}
\label{PWexact5}
\sum_{v\in V}f(v)=\sum_{u\in U}f(u)\sigma_{ u},\> U=V\setminus S, 
\end{equation}
\end{thm}
\begin{proof}

First, we show that the set $U$ is a uniqueness set for the space $\textbf{E}_{\omega}(\mathcal{L})$, i. e. for any two functions
from $\textbf{E}_{\omega}(\mathcal{L})$ the fact that they coincide on $U$ implies
that they coincide on $V$.

If $f,g\in \textbf{E}_{\omega}(\mathcal{L})$ then $f-g\in \textbf{E}_{\omega}(\mathcal{L})$ and
according to the  inequality (\ref{Bern}) the following 
holds true
\begin{equation}
\|\mathcal{L}(f-g)\|\leq \omega\|f-g\|.
\end{equation}
If $f$ and $g$ coincide on $U=V\backslash S$ then $f-g$ belongs
to  $L_{2}(S)$ and since $S$ is a $\Lambda$-set then we will have
$$
\|f-g\|\leq \Lambda\|\mathcal{L}(f-g)\|, \>\>f-g\in L_{2}(S).
$$
 Thus, if $f-g$ is not zero and $\omega<1/\Lambda$ we have the
following inequalities
\begin{equation}
\|f-g\|\leq \Lambda\|\mathcal{L}(f-g)\|\leq \Lambda
\omega\|f-g\|<\|f-g\|,
\end{equation}
which contradict to the assumption that $f-g$ is not identical
zero. Thus,  the set $U$ is a uniqueness set for the space $\textbf{E}_{\omega}(\mathcal{L})$.

It implies that there exists a constant $C=C(U,\omega)$ for which the following 
Plancherel-Polya inequalities hold true
\begin{equation}
\label{PP}
\left(\sum_{u\in U}|f(u)|^{2}\right)^{1/2}\leq\|f\|\leq
C\left(\sum_{u\in U}|f(u)|^{2}\right)^{1/2}
\end{equation}
for all $f\in \textbf{E}_{\omega}(\mathcal{L})$.  
Indeed,   the
functional
$$
|||f|||=\left(\sum_{u\in U}|f(u)|^{2}\right)^{1/2}
$$
defines another  norm on $\textbf{E}_{\omega}(\mathcal{L})$ because  the condition
$|||f|||=0, f\in \textbf{E}_{\omega}(\mathcal{L})$, implies that $f$ is identical
zero on entire graph. Since in finite-dimensional situation any two norms are equivalent we obtain  existence of a constant $C$ for
which (\ref{PP}) holds true.

 Let
$\delta_{v}\in L_{2}(G)$ be a Dirac measure supported at a vertex
$v\in V$. The notation $\vartheta_{v}$ will be used for a function
which is orthogonal projection of the function
$$
\frac{1}{\sqrt{d(v)}}\delta_{v}
$$
on the subspace $\textbf{E}_{\omega}(\mathcal{L})$. If $\varphi_{0}, \varphi_{1},...,\varphi_{j(\omega)}$ are orthonormal eigenfunctions of $\mathcal{L}$ which constitute an orthonormal basis in $\textbf{E}_{\omega}(\mathcal{L})$ then the explicit formula for $\vartheta_{v}$ is
\begin{equation}
\label{proj}
\vartheta_{v}=\sum_{j=0}^{j(\omega)}\varphi_{j}(v)\varphi_{j}.
\end{equation}
In these notations the Plancherel-Polya
inequalities (\ref{PP}) can be written in the form
\begin{equation}
\label{frame}
\sum_{u\in
U}|\left<f,\vartheta_{u}\right>|^{2}\leq\|f\|^{2}\leq
C^{2}\sum_{u\in U}|\left<f,\vartheta_{u}\right>|^{2},
\end{equation}
where $f, \vartheta_{u}\in \textbf{E}_{\omega}(\mathcal{L})$ and
$\left<f,\vartheta_{u}\right>$ is the inner product in $L_{2}(G)$.
These inequalities mean that if $U$ is a uniqueness set for the
subspace $\textbf{E}_{\omega}(\mathcal{L})$ then the functions
$\{\vartheta_{u}\}_{u\in U}$ form a frame in the subspace
$\textbf{E}_{\omega}(\mathcal{L})$ and the tightness of this frame is
$1/C^{2}$.  This fact implies that there exists a    frame of functions
$\{\Theta_{u}\}_{u\in U}$ in the space $\textbf{E}_{\omega}(\mathcal{L})$ such that
the following reconstruction formula holds true for all $f\in
\textbf{E}_{\omega}(\mathcal{L})$
\begin{equation}
f(v)=\sum_{u\in U}f(u)\Theta_{u}(v), v\in V.
\end{equation}
By setting $\sigma_{u}=\sum_{v\in V}\Theta_{u}(v)$ one obtains (\ref{PWexact5}).

\end{proof}

Unfortunately this approach does not give any information about constant $C$ in (\ref{frame}) and it make realization of the Theorem \ref{exact1} problematic. We are going to utilize another approach to the Plancherel-Polya-type inequality which was developed in our paper \cite{PesPes2} and which produces explicit constant. 

We will use the following notion of the relative degree.  Given any subset $A \subset V(G)$ and $v \in V(G)$, we let 
$$
w_A(v) = \sum_{u \in A} w(u,v).
$$
 We note that $w_A(v) = 0$ iff there is no edge connecting $v$ and some element of $A$.

Let's introduce the following notations

$$
U=S_{0},\>\>\partial(S_{0})=S_{1},\>\>\partial (S_{0}\cup S_{1})=S_{2},...,\partial(S_{0}\cup....\cup S_{n-1})=S_{n}.
$$
Clearly, $\{S_{m}\}_{0}^{n}$ is a disjoint cover of $V(G)$.
We let
\[
 D_{m}  = D_m(\mathcal{S}) = \sup_{v \in S_m} w_{S_{m+1}}(v)
\]
and 
\[
 K_m = K_m (\mathcal{S}) = \inf_{v \in S_{m+1}} w_{S_m}(v).
\] 
For $0 \le m < n$ let
\[
 \widehat{K}_m(\widehat{\mathcal{S}}) = \widehat{K}_m = \inf_{v \in S_{m}} w_{S_{m+1}}(v),
\] 
as well as
\[
 \widehat{D}_m (\widehat{\mathcal{S}}) = \widehat{D}_m = \sup_{v \in S_{m+1}} w_{S_{m}}(v).
\]
The set $U=S_0$ is called {\em inital set} of the partition $\mathcal{S}$, it is of primary importance for the following results. 

We define 
\[
\delta_{U} = 
\left( \sum_{m=1}^n \left( \sum_{k=1}^m \frac{1}{K_{k-1}} \left( \prod_{i=k}^{m-1} \frac{D_i}{K_i} \right) \right) \right)^{1/2},
\]

\[
 a_{U} =   \left( \sum_{m=0}^{n} \prod_{j=0}^{m-1} \frac{D_j}{K_j} \right)^{1/2},
\]

\[
\hat{\delta}_{U} = 
\left( \sum_{m=1}^{n'} \left( \sum_{k=0}^{m-1} \frac{1}{\widehat{K}_{k}} \left( \prod_{i=k}^{m-1} \frac{\widehat{K}_i}{\widehat{D}_i} \right) \right) \right)^{1/2},
\]

\[
 \hat{a}_{U} =   \left( \sum_{m=0}^{n'} \prod_{j=0}^{m-1} \frac{\widehat{K}_j}{\widehat{D}_j} \right)^{1/2}.
\]
The weighted gradient norm of a function $f$ on $V(G)$ is defined by 
\[
 \| \nabla_w f \| = \left( \sum_{u, v \in V(G)} \frac{1}{2} |f(u) - f(v)|^2 w(u,v) \right)^{1/2}.
\] 
It is known  \cite{FP}, \cite{Mo}, that
\begin{equation}\label{L-G}
\|\mathcal{L}^{1/2}f\|=\|\nabla f  \|.
\end{equation}
Note,  that  if $f$ belongs to the space   ${\bf E}_{\omega}(\mathcal{L})$ 
then the Bernstein inequality gives 
\begin{equation}
\label{Bern}
\|\nabla f\|=\|\mathcal{L}^{1/2}f\|\leq \sqrt{\omega}\|f\|.
\end{equation}
After all these preparations we can formulate the following statement which follows from  \cite{FP}.

\begin{thm} \label{cor:samp_2}
If the inequality 
\begin{equation}
\label{PPcond}
\delta_{U}\sqrt{\omega}< 1,\>\>\>\omega>0,
\end{equation}
  is satisfied, then the following Plancherel-Polya-type  equivalence holds for all $f\in \textbf{E}_{\omega}(\mathcal{L})$: 
\begin{equation} \label{eqn:cor_norm_equiv_sharp}
 \frac{1-\delta_{U}\sqrt{\omega} }{a_{U}} \| f \| \le \| f |_U \| \le \frac{1+\hat{\delta}_{U}\sqrt{\omega}}{\hat{a}_{U}} \| f \|.
\end{equation}
\end{thm}

Using this result we prove existence of exact cubature formulas on spaces of bandlimited functions.
\begin{thm}
\label{exact20}
If  $U$ is a subset of vertices  for which the inequality ( \ref{PPcond}) is satisfied  then there exists a set of weights $\mu_{u}\in \mathbb{R}, \>u\in U,$ such that for any $f\in \textbf{E}_{\omega}(\mathcal{L})$, where $\omega$ satisfies   (\ref{PPcond}) 
the following exact formula holds
\begin{equation}
\label{PWexact50}
\sum_{v\in V}f(v)=\sum_{u\in U}f(u)\mu_{ u}.
\end{equation}
\end{thm}
\begin{proof}
The previous Theorem shows that $U$ is a uniqueness set for the space $\textbf{E}_{\omega}(\mathcal{L})$, which means that every $f$ in $\textbf{E}_{\omega}(\mathcal{L})$ is uniquely determined by its values on $U$.

 Let us denote by  $\theta_{v}$, where $v\in U$, the orthogonal projection of the Dirac measure $\delta_{v},\>\>v\in U$, onto the space $\textbf{E}_{\omega}(\mathcal{L})$. Since for functions in $\textbf{E}_{\omega}(\mathcal{L})$ one has $f(v)=\left<f, \theta_{v}\right>,\>\>v\in U$,   the inequality (\ref{eqn:cor_norm_equiv_sharp}) takes the form of a frame inequality in the Hilbert space $H=\textbf{E}_{\omega}(\mathcal{L})$
\begin{equation}
\label{frame-ineq}
\left( \frac{1-\epsilon \delta_{U}}{a_{U}} \right)^{2}\| f \|^{2} \le \sum_{v\in U}|\left<f,\theta_{v}\right>|^{2}                   \le \left(\frac{1+\epsilon \hat{\delta}_{U}}{\hat{a}_{U}}\right)^{2} \| f \|^{2},\>\>\epsilon=\sqrt{\omega},
\end{equation}
for all $\>\>f\in \textbf{E}_{\omega}(\mathcal{L})$. According to the general theory of Hilbert frames \cite{Gr} the last inequality implies that there exists a dual frame (which is not unique in general) $\{\Theta_{v}\},\>v\in U,\>\>\Theta_{v}\in \textbf{E}_{\omega}(\mathcal{L})$, in the space $\textbf{E}_{\omega}(\mathcal{L})$ such that for all $f\in \textbf{E}_{\omega}(\mathcal{L})$ the following reconstruction formula holds 
\begin{equation}
\label{reconstruction}
f=\sum_{v\in U}f(v)\Theta_{v}.
\end{equation}
By setting $\sum_{v\in V}\Theta_{v}(u)=\mu_{u}$ we obtain (\ref{PWexact50}).
\end{proof}

To be more specific we consider unweighted  case for which
\begin{equation}
\label{cl}
U=U\cup \partial(U)=V(G).
\end{equation}
In other words, we consider a bipartite graph $V(G)$ with components $S_{0}=U$ and $S_{1}=\partial(U)$.
Keeping the same notations as above we compute
$$
a_{U}=\left(1+\frac{D_{0}}{K_{0}}\right)^{1/2},\>\>\>
\delta_{U}=\frac{1}{K_{0}^{1/2}}.
$$
Thus, we have 
$$
\|f\|\leq \left(1+\frac{D_{0}}{K_{0}}\right)^{1/2}\|f_{0}\|+\frac{1}{K_{0}^{1/2}}\|\nabla f\|.
$$
By applying (\ref{Bern}) along with assumption
\begin{equation}
\label{condition}
\omega<K_{0}
\end{equation}
 we obtain the following estimate 
\begin{equation}
\label{rightestimate}
\|f\|\leq \left(1-\sqrt{\frac{\omega}{K_{0}}  }\right)\left(1+\frac{D_{0}}{K_{0}}\right)^{1/2}\|f_{0}\|,\>\>\>f_{0}=f|_{U}.
\end{equation}
On the other hand 
$$
\hat {a}_{U}=\left(1+\frac{\hat{K}_{0}}{\hat{D}_{0}}\right)^{1/2},\>\>\>
\hat{\delta}_{U}=\frac{1}{\hat{D}^{1/2}_{0}}.
$$
This yields the norm estimate
$$
\|f\|+\frac{1}{\hat{D}_{0}^{1/2}}\|\nabla f\|\geq \left(1+\frac{\hat{K}_{0}}{\hat{D}_{0}}\right)^{1/2}\|f_{0}\|,\>\>\>f_{0}=f|_{U}.
$$
If (\ref{Bern}) holds, then
\begin{equation}
\label{leftside}
 \left(1+\frac{\hat{K}_{0}}{\hat{D}_{0}}\right)^{1/2}\|f_{0}\|\leq \|f\|_{2}+\frac{1}{\hat{D}_{0}^{1/2}}\|\nabla f\|\leq \left(1+\sqrt{\frac{\omega}{\hat{D}_{0}}}\right)\|f\|.
\end{equation}
After all, for functions $f$ in $\textbf{E}_{\omega}(\mathcal{L})$ with $\omega <K_{0}$ we obtain the following frame inequality

\begin{equation}
A\|f\|^{2}\leq \sum_{v\in U}|<f,\theta_{v}>|^{2}     \leq B\|f\|^{2},\>\>\>f_{0}=f|_{U},
\end{equation}
where
\begin{equation}
\label{A}
A=\frac
{    \left(1-\sqrt{\frac {\omega}{K_{0}}}\right)^{2} }
{        1+\frac{D_{0}}{K_{0}}   }, \>\>\>
B=\frac
{    \left(1+\sqrt{\frac {\omega}{\widehat{D_{0}}}}\right)^{2} }
{        1+\frac{\widehat{K_{0}}}{\widehat{D_{0}}}  }.
\end{equation}
It shows that if the condition $\omega<K_{0},\>\>K_{0}=K_{0}(U),$ is satisfied then the set $U$ is a sampling set for the space $\textbf{E}_{\omega}(\mathcal{L})$ and a reconstruction formula (\ref{reconstruction}) holds which leads to (\ref{PWexact50}).

Using the same notations as above we summarize these observations on the following statement. 
\begin{thm}
If $G$ is a bipartite graph with components $S_{0}$ and $S_{1}$, 
$$
K_{0}=\inf_{v\in S_{1}}w_{S_{0}}(v)
$$
then
\begin{enumerate}

\item $S_{0}$ is a uniqueness set for ${\bf E}_{\omega}(\mathcal{L})$ for any $\omega<K_{0}$;

\item if $\theta_{u}$ is orthogonal projection of $\delta_{u}$ onto ${\bf E}_{\omega}(\mathcal{L}),\>\>\omega<K_{0},\>$ then $\>\{\theta_{u}\}_{u\in S_{0}}$ is a frame in ${\bf E}_{\omega}(\mathcal{L})$ with constants (\ref{A});

\item if $\{\Theta_{u}\}$ is a frame dual in ${\bf E}_{\omega}(\mathcal{L}), \>\>\omega<K_{0},\>$ to the frame $\>\{\theta_{u}\}_{u\in S_{0}}$ and 
$$
\sum_{v\in V}\Theta_{u}(v)=\mu_{u},\>\>\>u\in S_{0},
$$
then for any $f\in \textbf{E}_{\omega}(\mathcal{L})$
the following exact formula holds
\begin{equation}\label{last-100}
\sum_{v\in V}f(v)=\sum_{u\in S_{0}}f(u)\mu_{ u}.
\end{equation}

\end{enumerate} 

\end{thm}

Note, that if a bipartite graph $G=S_{0}\cup S_{1}$ is complete and $|S_{0}|=N>M=|S_{1}|$ then the spectrum of $\mathcal{L}$ consists of $0,\> M,\> N, \>M+N,$ where $M $ has multiplicity $N-1$ and $N$ has multiplicity $M-1$. In this case $K_{0}=N$ and according to the last theorem formula (\ref{last-100}) integrates exactly all the functions from the $N$-dimensional subspace ${\bf E}_{\omega}(\mathcal{L}),\>\>\omega<N,$ of the $N+M$-dimensional space $L_{2}(G)$.


\begin{thebibliography}{22}

\bibitem {Ch}
F. R. K. ~Chung, {\em Spectral Graph Theory}, CBMS 92, AMS, 1994.



\bibitem{CM} R.~Coifman, M.~Maggioni, 
{\em Diffusion wavelets for multiscale analysis on graphs and manifolds,} in {\em Wavelets and splines: Athens 2005}, pp. 164-188, Mod. Methods Math., Nashboro Press, Brentwood, TN, 2006.  


\bibitem{GNC}
M. ~Gavish, B. ~Nadler, R. ~Coifman, {\em Multiscale wavelets on trees, graphs and high dimensional data: Theory and applications to
semi-supervised learning}, ICML 2010.

\bibitem {Gr}
K.~Gr\"ochenig, {\em Foundations of time-frequency analysis},
Birkhauser, 2001.

\bibitem{FP}
H. ~ F\"uhr,  I.  Z. Pesenson, {\em     PoincarŽ and Plancherel-Polya inequalities in harmonic analysis on weighted combinatorial graphs }, accepted for publication by SIDMA.


 \bibitem{HVG}
 D. Hammond, P.  Vandergheynst, R. Gribonval, {\em Wavelets on graphs via spectral graph theory},  Appl. Comput. Harmon. Anal. 30 (2011), no. 2, 129Ð150.
 
 

\bibitem{Mo} B. Mohar, {\em Some applications of Laplace eigenvalues of graphs}, in G. Hahn and G. Sabidussi, editors, {\em Graph Symmetry: Algebraic Methods and Applications (Proc. Montr\'real 1996)}, volume 497 of {\rm Adv. Sci. Inst. Ser. C. Math. Phys. Sci.}, pp. 225-275, Dordrecht (1997), Kluwer.


\bibitem{NSZ}
B. Nadler, N. Srebro and X. Zhou, {\em Semi-supervised Learning with the Graph Laplacian}, NIPS 2009.


\bibitem{NO}
Sunil K. Narang and Antonio Ortega, {\em Compact Support Biorthogonal Wavelet Filterbanks for Arbitrary Undirected Graphs},  (2012), arXiv:1210.8129.

\bibitem{Pes08}
 I.~Pesenson, {\em  Sampling in Paley-Wiener spaces on combinatorial graphs},
 Trans. Amer. Math. Soc. 360 (2008), no. 10, 5603--5627.



\bibitem{Pes09}
 I.~Pesenson, {\em  Variational splines and Paley-Wiener spaces on combinatorial graphs},
Constr. Approximation,  29 (2009), no. 1, 1--20.

\bibitem{Pes10}
I.~Pesenson, {\em Removable sets and eigenvalue and eigenfunction
approximations on finite combinatorial graphs}, Applied and Computational Harmonic Analysis,  29  (2010),  no. 2, 123-133.

\bibitem{PesPes1}
I.~Pesenson, M. ~Pesenson, {\em Eigenmaps and minimal and
bandlimited immersions of graphs into Euclidean spaces},
J. of Mathematical Analysis and Applications,  366  (2010),  no. 1, 137-152.

\bibitem{PesPes2}
I.~Pesenson, M. ~Pesenson, {\em Sampling, filtering and sparse approximations on combinatorial graphs}, J. Fourier Anal. Appl.  16  (2010),  no. 6, 921-942.

\bibitem{PG}
I. Z. Pesenson, D. ~Geller, 
{\em Cubature formulas and discrete Fourier transform on compact
manifolds} in  "From Fourier
Analysis and Number Theory to Radon Transforms and
Geometry: In Memory of Leon Ehrenpreis" (Developments in
Mathematics 28) by H.M. Farkas, R.C. Gunning,
M.I. Knopp and B.A. Taylor, Springer NY (2013).


 \bibitem{SNFOV}
 David I Shuman, Sunil K. Narang, Pascal Frossard, Antonio Ortega, Pierre Vandergheynst, 
{\em The Emerging Field of Signal Processing on Graphs: Extending High-Dimensional Data Analysis to Networks and Other Irregular Domains}, 
arXiv:1211.0053

\end{thebibliography}
\end{document}